\documentclass[11pt,a4paper,reqno]{amsart}
\usepackage[T1]{fontenc}
\usepackage{amsmath}
\usepackage{amsfonts}
\usepackage{amssymb}
\usepackage{amsbsy}
\usepackage{mathrsfs}
\usepackage{bbm}
\usepackage{a4wide}
 
\newtheorem{theorem}{Theorem} 
\newtheorem{lemma}[theorem]{Lemma}

\newtheorem{proposition}[theorem]{Proposition}

\newtheorem{corollary}{Corollary}

\theoremstyle{remark}
\newtheorem{remark}[theorem]{\it \bf{Remark}\/}

\catcode`@=11
\def\section{\@startsection{section}{1}%
  \z@{1.5\linespacing\@plus\linespacing}{.5\linespacing}%
  {\normalfont\bfseries\large\centering}}
\catcode`@=12
\newcommand{\be}{\begin{equation}}
\newcommand{\ee}{\end{equation}}
\newcommand{\bea}{\begin{eqnarray}}
\newcommand{\eea}{\end{eqnarray}}
\newcommand{\bee}{\begin{eqnarray*}}
\newcommand{\eee}{\end{eqnarray*}}

\def\pa{\partial}

\def\RR{\mathbb{R}}

\def\fref#1{{\rm (\ref{#1})}}

\catcode`@=11
\def\supess{\mathop{\operator@font Sup\,ess}}
\catcode`@=12

\def\e{\varepsilon}

\def\bar#1{{\overline #1}}
\def\fref#1{{\rm (\ref{#1})}}

\def\R2+{\RR ^2_+}

\def\lsl{\frac{\lambda_s}{\lambda}}

\def\pa{\partial}

\def\lim{\mathop{\rm lim}}

\def\sup{\mathop{\rm sup}}

\def\l{\lambda}

\def\log{{\rm log}}

\def\lsl{\frac{\lambda_s}{\lambda}}
\def\xsl{\frac{x_s}{\lambda}}

\def\matchal{\mathcal}

\def\pa{\partial}

\def\pa{\partial}

\def\R{\Bbb R}

\title[Singularity formation for critical problems]{Near soliton dynamics and singularity formation for $L^2$ critical problems}
\author[Y. Martel]{ {Yvan Martel}}
\address{Ecole Polytechnique, CMLS UMR7640}
\email{yvan.martel@polytechnique.edu}
\author[F. Merle]{Frank Merle}
\address{Universit\'e de Cergy Pontoise and Institut des Hautes \'Etudes Scientifiques, AGM CNRS UMR8088}
\email{merle@math.u-cergy.fr}
\author[P. Rapha\"el]{Pierre Rapha\"el}
\address{Universit\'e de Nice Sophia Antipolis and Institut Universitaire de France, LJAD CNRS UMR7361 }
\email{praphael@unice.fr}
\author[J. Szeftel]{J\'er\'emie Szeftel}
\address{CNRS and 
Université Pierre et Marie Curie , LJLL UMR7598}
\email{jeremie.szeftel@upmc.fr}
\begin{document}

\begin{abstract}
This survey reviews  the state of the art concerning the singularity formation for two canonical   dispersive problems: the mass critical non linear Schr\"odinger equation and the mass critical  generalized  KdV equation. In particular, we address the question of the classification of the flow for initial data  near the soliton.
\end{abstract}
\maketitle 


\section{Introduction}


The study of singularity formation in nonlinear dispersive equations has attracted considerable attention for the past thirty years. Recently, this activity  has led to the development of robust tools to construct and describe blow up solutions, solving some of the classical conjectures in the field. Description of blow up solutions typically  includes determination of   blow up speed, blow-up profile, behavior of concentration  points. 
We start by introducing below the two main equations to be considered in this review.

\subsection{The $L^2$ critical (NLS) and (gKdV) problems} Our aim in these notes is to review some recent progress done in the past ten years on singularity formation for two canonical problems: the $L^2$ critical non linear Schr\"odinger equation 
\be
\label{nls}
{\rm (NLS)} \qquad \ \left\{\begin{array}{ll} i\pa_t u +\Delta u +u|u|^{\frac 4d}=0\\ u_{|t=0}=u_0\end{array}\right . \qquad \ (t,x)\in \Bbb R\times \Bbb R^d, \quad u\in \Bbb C,
\ee
and the $L^2$ critical (one dimensional) generalized Korteweg--de Vries equation
\be
\label{kdv}
{\rm (gKdV)} \qquad \ \left\{\begin{array}{ll} \pa_t u +(u_{xx}+u^5)_x=0\\ 
u_{|t=0}=u_0\end{array}\right . \qquad \ (t,x)\in \Bbb R\times \Bbb R, \quad u\in \Bbb R.
\ee

A specific algebra underlies  these two models. Solutions of both models preserve the so--called energy $$E(u)=\frac12\int_{\Bbb R^d}|\nabla u|^2-\frac{1}{2+\frac 4d}\int_{\Bbb R^d}|u|^{2+\frac{4}{d}}=E(u_0)$$ and the mass $$\int_{\Bbb R^d}|u|^2=\int_{\Bbb R^d}|u_0|^2$$ 
(take $d=1$ in both expressions in the case of gKdV). The notion of $L^2$ criticality means that  the scaling symmetry of the equation:
$$\hbox{if $u(t,x)$ is solution then} \ u_{\l}(t,x):=\left\{\begin{array}{ll}\l^{\frac d2}u(\l^2t,\l x) \hbox{ for (NLS)}\\ \l^\frac{1}{2}u(\l^3 t,\l x)   \mbox{ for (gKdV)} \end{array}\right.\\ \hbox{is also solution,}$$
  leaves the $L^2$ norm unchanged.

\subsection{Local existence theory and   blow up}

From  Ginibre and Velo \cite{GV}, (\ref{nls}) is locally well-posed in $H^1(\R^d)$, and from  Kenig, Ponce and Vega \cite{KPV}, \eqref{kdv}  is locally well-posed in $H^1(\R)$. Thus, for any $u_0\in H^1$, there exists $0<T\leq +\infty$ and a unique maximal solution $u(t)\in {\mathcal{C}}([0,T),H^1)$ to either (\ref{nls}) or \eqref{kdv}, with the following alternative: 
\begin{itemize}
\item either $T=+\infty$, the solution is globally defined in $H^1$;

\item or $T<+\infty$ and then the solution blows up in finite time:
$$\lim_{t\uparrow T}\|\nabla u(t)\|_{L^2}=+\infty.$$
\end{itemize}

\subsection{Ground state and sharp threshold for global existence} Exceptional solutions   propagating without deformation (called traveling waves, or solitons) play a distinguished role in the analysis. Indeed, the ansatz 
\be
\label{groudnstate}
u(t,x)=\left\{\begin{array}{ll}Q(x)e^{it}\ \ \mbox{for  (NLS)}\\ Q(x-t) \ \ \mbox{for (gKdV)}\end{array}\right. 
\ee leads to the semi linear elliptic problem $$\Delta Q-Q+Q^{1+\frac{4}{d}}=0,\quad x\in \Bbb R^d,$$ which admits a {unique} up to space translation,  {positive}, $H^1$ solution known as the ground state solitary wave. Explicit variational characterization of $Q$ related to sharp Sobolev bounds led in the 80's to the first derivation of  a sharp criterion of global existence versus blow up for (NLS)  by Weinstein \cite{W1983} and Berestycki, Cazenave \cite{BeresCaze}. In particular, $H^1$ initial data with $$\|u_0\|_{L^2}<\|Q\|_{L^2}$$   for both (NLS) and (gKdV) generate a unique global solution $u\in \mathcal C([0,+\infty, H^1)$. In the case of (NLS), it is moreover known in this case that the solution scatters, {i.e.} behaves like a free wave as $t\to+\infty$: $$\exists u_{\pm\infty}\in H^1\ \ \mbox{such that}\ \ \lim_{t\to \pm\infty}\|u(t)-e^{it\Delta} u_{\pm\infty}\|_{L^2}=0.$$ We refer to \cite{Cbook} and to the  recent definitive result \cite{Dodson} (see also references therein). In other words, the ground state solitary wave is the {\it smallest non dispersive and thus non linear object}.\\

In many nonlinear problems, solitary waves are expected to be the building blocks for nonlinear dynamics: for large time any solution decomposes into a certain number  of   solitary waves plus a residual which is a free wave. 
This fact is well-known for integrable problems from the 60's, at least at the formal level.
In blow up regimes,  the role played by   solitary waves has also been clarified   in several directions: $Q$ is the universal blow up profile for blow up solutions initially close to $Q$, independently of the blow up speed. It is thus a fundamental object to study the  formation of singularity in such situations.
Moreover, the complete description of the  non linear flow near the ground state solitary wave has become one of the most relevant and challenging   question in the field.\\

Theses notes are organized as follows.
In Section 2, we summarize  results on blow up solutions near $Q$ obtained for (NLS) in \cite{MRgafa}, \cite{MRinvent}, \cite{MRannals}, \cite{MRjams},  \cite{MRS}, \cite{RS2010}. In Section 3, we discuss for (gKdV) the   description of the flow near the ground state obtained in \cite{MMR1}, \cite{MMR2}, \cite{MMR3}. 
Finally, an overview of some proofs for (gKdV) are given in Section \ref{sec:kdvproof}.


\section{The $L^2$ critical (NLS) problem}\label{sec:massnls}


In this section, we review  the state of the art on the problem of formation of singularities for the $L^2$ critical (NLS), and present some other results on related problems.


\subsection{Minimal mass blow up}\label{sec:mininalblowup}


As discussed in the Introduction, initial data $u_0\in H^1$ with  mass $\|u_0\|_{L^2}<\|Q\|_{L^2}$ generate global bounded solutions. For (NLS), this criterion is   sharp, as a   consequence of the so--called pseudo conformal transformation which is a well-known symmetry of the linear Schr\"odinger flow and of (NLS) in the $L^2$ critical case: if $u(t,x)$ is a solution to (NLS), then so is   
\be
\label{transfconf}
v(t,x)=\frac{1}{|t|^{\frac{d}{2}}}u\left(-\frac{1}{t},\frac{x}{t}\right)e^{i\frac{|x|^2}{4t}}.
\ee 
 Applied to the solitary wave solution $u(t,x)=Q(x)e^{it}$, one gets an {\it explicit minimal mass blow up solution}:
\be
\label{st}
S_{\rm NLS}(t,x)=\frac{1}{|t|^{\frac{d}{2}}}Q\left(\frac{x}{t}\right)e^{i\frac{|x|^2}{4t}-\frac{i}{t}},\qquad
\|S_{\rm NLS}(t)\|_{L^2} =\|Q\|_{L^2}.
\ee
The dynamics generated by the smooth data $S_{\rm NLS}(-1)$ is explicit: $S_{\rm NLS}(t)$ {scatters}  as $t\to -\infty$, and blows up as $t\uparrow 0$ at the speed 
\be
\label{sppedt}
\|\nabla S_{\rm NLS}(t)\|_{L^2}\sim \frac{1}{|t|}.
\ee
An essential feature of \eqref{st} is that $S_{\rm NLS}(t)$ is {\it compact} up to the symmetries of the flow, meaning that all the mass goes into the singularity formation
\be
\label{concn}
|S_{\rm NLS}(t)|^2\rightharpoonup \|Q\|_{L^2}^2\delta_{x=0}\ \ \mbox{as}\  \ t\uparrow 0.
\ee
The general intuition is that such a behavior is exceptional in the sense that such minimal elements can be classified\footnote{This is a dispersive intuition which for example is wrong in the parabolic setting, \cite{BCM}.}. 
The first result of this type was proved by Merle using the pseudo conformal symmetry.

\begin{theorem}[Classification of the minimal mass blow up solution,  \cite{Mduke}]
\label{critmass}
Let $u_0\in H^1$ with $$\|u_0\|_{L^2}=\|Q\|_{L^2}.$$ Assume that the corresponding solution to (NLS) blows up in finite time $0<T<+\infty$. Then $$u(t)=S_{\rm NLS}(t)$$ up to the symmetries of (NLS).
\end{theorem}

The question of   existence of minimal elements in various other settings has been a long standing open problem, mostly due to the fact that the existence of the minimal element for (NLS) relies entirely on the exceptional pseudo conformal symmetry. Merle in \cite{Merlenonexistence} considered the inhomogeneous problem 
\begin{equation}\label{nlsk:0}
i\pa_tu+\Delta u+k(x)u|u|^2=0, \ \ x\in \Bbb R^2 
\end{equation}
which breaks the full symmetry group, and obtains for non smooth $k$ {\it non existence} results of minimal elements.  A contrario and more recently, a sharp criterion on the inhomogeneity $k(x)$ for the existence and uniqueness of minimal solutions is derived  in  \cite{RS2010} (see also \cite{BCD} concerning   existence). 

\begin{theorem}[Existence and uniqueness of a critical element for \eqref{nlsk:0},  \cite{RS2010}]
\label{th:inhomogeneous}
Let $x_0\in \RR^2$ with $$k(x_0)=1 \ \ \mbox{and} \ \ \nabla^2k(x_0)<0.$$ 
Let the energy $E_0$ of $u$ satisfy:
\be\label{condblowup}
E_0+\frac{1}{8}\int \nabla ^2k(x_0)(y,y)Q^4>0.
\ee
Then, there exists a critical mass $H^1$ blow up solution to \eqref{nlsk:0}, unique up to phase shift,  which blows up at time $T=0$ and at the point $x_0$ with energy $E_0$. Moreover, 
\be
\label{mommentgotozero}
\lim_{t\to 0}Im\left(\int\nabla u\overline{u}\right)=0.
\ee
\end{theorem}

Under condition \eqref{condblowup},  minimal blow up elements at a non degenerate blow up point are thus completely classified.
It is  also shown in \cite{RS2010} that  \eqref{condblowup} is a necessary condition for blow-up.
 Theorem \ref{th:inhomogeneous} relies on a dynamical construction and new Lyapunov   functionals at the minimal mass level. A further extension to non local dispersion $$i\pa_tu-(-\Delta)^{\frac 12}u+u|u|^2=0, \ \ x\in \R$$  can be found in \cite{KLR2}, see also \cite{boulengerthesis} for an extension to curved backgrounds. These works show that the existence of a minimal mass bubble is a general property, independent of the exceptional existence of a pseudo conformal symmetry for the model.  


\subsection{log-log blow up}\label{sec:loglog}


The minimal mass blow up solution \fref{st} is explicit, but obviously the corresponding blow up scenario is unstable since any subcritical mass perturbation of $S_{\rm NLS}(t)$   leads to a globally defined solution. The question of the description of {\it stable} blow up bubbles has attracted a considerable attention which started in the 80's with the development of sharp numerical methods and the prediction of the ``log-log law'' for NLS by Landman, Papanicoalou, Sulem, Sulem \cite{PSS}.\\

We focus our attention to initial data with mass slightly above the minimal mass required for singularity formation 
\be
\label{smallamss}
u_0\in \mathcal B_{\alpha^*}=\left\{u_0\in H^1\ \ \mbox{with}\ \ \|Q\|_{L^2}<\|u_0\|_{L^2}<\|Q\|_{L^2}+\alpha^*\right\}, \ \ 0<\alpha^*\ll 1.
\ee 
Applying  concentration-compactness techniques \cite{PLL1} and  using the variational characterization of the ground state,  assumption  \eqref{smallamss} implies that if $u(t)$ blows up at $T<+\infty$, then for $t$ close   to $T$, the solution admits a nonlinear decomposition 
\be
\label{decompou}
u(t,x)=\frac{1}{\lambda(t)^{\frac{d}{2}}}(Q+\e)\left(t,\frac{x-x(t)}{\lambda(t)}\right)e^{i\gamma(t)},
\ee
where 
\be
\label{deflambda} \lambda(t)\sim \frac{1}{\|\nabla u(t)\|_{L^2}},\quad 
\|\e(t)\|_{H^1}\leq \delta(\alpha^*),\ \ \lim_{\alpha^*\to 0} \delta(\alpha_*)=0 .
\ee 
This decomposition implies that independently of the blow up regime, the ground state solitary wave $Q$ is  a good approximation of the blow up profile, which is the starting point of a perturbative analysis for \eqref{smallamss}. {\it The sharp description of the blow up bubble now relies on the determination  of a finite dimensional   dynamical system for a suitable choice of   geometrical parameters $(\lambda(t), x(t), \gamma(t))$, coupled to  the infinite dimensional   dynamics driving the   residual term $\e(t)$.}

\begin{remark} For example, one can decompose the  minimal mass blow up solution  \eqref{st}  as follows: $$\lambda(t)=|t|,\ \ \e(t,y)=Q(y)\left(e^{i\frac{b(t)|y|^2}{4}}-1\right), \ \ b(t)=|t|.$$
\end{remark}

All possible regimes for $(\lambda(t),x(t),\gamma(t))$ are not known so far for (NLS), but   progress has been made on the understanding of both  ``stable'' and ``threshold'' dynamics. The following theorem summarizes a series of results obtained in \cite{MRannals}, \cite{MRgafa}, \cite{MRinvent}, \cite{MRjams}, \cite{MRcmp}, \cite{Rannalen}:

\begin{theorem}[\cite{MRannals}, \cite{MRgafa}, \cite{MRinvent}, \cite{MRjams}, \cite{MRcmp}, \cite{Rannalen}]
\label{theoremdeux}
Let $d\leq 5$. There exists a universal constant $\alpha^*>0$ such that the following holds true. Let $u_0\in \mathcal B_{\alpha^*}$ and $u\in \mathcal C([0,T),H^1)$, $0<T\leq +\infty$ be the corresponding maximal solution to \eqref{nls}.\\
{\rm (i) Sharp $L^2$ concentration:} Assume $T<+\infty$. Then there exist parameters
 $(\lambda(t), x(t), \gamma(t))\in \mathcal C^1([0,T),\R_+^*\times\R^d\times \R)$ and an asymptotic profile $u^*\in L^2$ such that 
\be
\label{strignltwo}
u(t)-\frac{1}{\lambda(t)^{\frac{d}{2}}}Q\left(\frac{x-x(t)}{\lambda(t)}\right)e^{i\gamma(t)}\to u^* \ \ \mbox{in} \ \ L^2\ \ \mbox{as} \ \ t\to T.
\ee
Moreover, the blow up point is finite: $$x(t)\to x(T)\in \R^d \ \ \mbox{as} \ \ t\to T.$$ 
{\rm (ii) Blow up speed:} Under {\rm (i)}, the following alternative holds:

- either the solution satisfies the {\rm log-log regime}, i.e. 
\be
\label{logloglaw}
\lambda(t)\sqrt{\frac{\log|\log(T-t)|}{T-t}}\to \sqrt{2\pi} \ \ \mbox{as} \ \ t\to T
\ee

and then the asymptotic profile is not smooth:
\be
\label{ustarnotlp}
u^*\notin H^1 \ \ \mbox{and} \ \ u^*\notin L^p \ \ \mbox{for all} \ \ p>2;
\ee

- or there holds the sharp lower bound, for $t$ close to $T$,
 \be
 \label{lowerbound}
\lambda(t)\leq C(u_0)(T-t), \quad \hbox{equivalently} \quad \| \nabla u(t)\|_{L^2} \geq \frac {C_1(u_0)}{T-t}
 \ee
 
and then the asymptotic profile satisfies
\be
\label{ustarsmooth}
u^*\in H^1.
\ee
{\rm (iii) Sufficient condition for log-log blow up:} Assume $E_0\leq 0$. Then the solution blows un finite time $T<+\infty$ in the log-log regime \eqref{logloglaw}.\\
{\rm (iv)  $H^1$ stability of the log-log blow up:} The set of initial data in $\mathcal B_{\alpha^*}$ such that the corresponding solution to (\ref{nls}) blows up in finite time with the log-log law (\ref{logloglaw}) is open in $H^1$.\end{theorem}

\noindent {\bf Comments on the result.} 

{\it 1. The log-log law}. The stable blow log-log law \eqref{logloglaw}   was   proposed in the pioneering formal and numerical work \cite{PSS}. The first rigorous construction of  a  class of solutions with such blow up speed is due to G. Perelman \cite{PE} in dimension $d=1$. The proof of Theorem \ref{theoremdeux} involves a mild coercivity property of the linearized operator close to $Q$, which is proved in dimension $d=1$ in \cite{MRannals} and checked numerically in an elementary way in \cite{FMR} for $d\leq 5$. For   dimensions $d\geq 2$, the lack of explicit formula for the ground state is a difficulty to prove this property.\\

{\it 2. Upper bound on the blow up speed}: No general upper bound on the blow up speed $\|\nabla u(t)\|_{L^2}$ is known in the $L^2$ critical case, even for data $u_0\in \mathcal B_{\alpha^*}$. This is in contrast with super critical regimes where recently a sharp upper bound has been derived (see \cite{MRSring} and section~\ref{sec:ringblowup}). The lower bound \eqref{lowerbound} is sharp and attained by the minimal blow up element $S_{\rm NLS}(t)$. The derivation of different blow up speeds, which is equivalent  to the construction of infinite time grow up solutions through the pseudo conformal symmetry, is related to the description of the flow near the ground state, still incomplete for (NLS). Some intuition may come from the recent classification results obtained for the $L^2$ critical gKdV problem and presented in section \ref{sectionkdv}.\\

{\it 3. Quantization of the blow up mass}: The strong convergence \eqref{strignltwo}  describes precisely the blow up bubble in the scaling invariant space and implies in particular that the mass  put into the singularity is quantized $$|u(t)|^2\rightharpoonup \|Q\|_{L^2}^2\delta_{x=x(T)}+|u^*|^2\ \ \mbox{as} \ \ t\to T, \ \ |u^*|^2\in L^1.$$ Such quantization and the convergence result \eqref{strignltwo} rely on   the property of asymptotic stability  of the solitary wave in blow up regime, which writes as follows  in the formulation \eqref{decompou} $$\e(t,x)\to 0 \ \ \mbox{as}\ \ t\to T\ \ \mbox{in}\ \ L^2_{\rm loc}.$$ In fact, in proving  Theorem \ref{theoremdeux}, the derivation of either upper bounds or lower bounds on the blow up rate is closely related to the question of dispersion for the residual $\e(t,x).$ 
Such asymptotic  theorems started to appear in the dispersive setting in \cite{MMannals}, and significant progress was  made in a recent classification result - without any assumption of size on the data - for energy critical wave equations \cite{KDM1}. \\

{\it 4. Asymptotic profile}: The regularity of the asymptotic profile $u^*$ depends directly upon the regime because    singular and regular parts of the solution are very much coupled in the stable log-log regime, while they are   weakly interacting in any other regime.


\subsection{Threshold dynamics}
\label{bw}


Theorem \ref{theoremdeux} describes the stable log-log blow up in the neighborhood of the soliton, but  it does not complete the study of blow up for (NLS) even for initial data $u_0\in \mathcal B_{\alpha^*}$. In particular, it remains to  clarify   unstable blow up close to $Q$. The explicit minimal mass blow up solution $S_{\rm NLS}(t)$ defined in \eqref{theoremdeux}   is clearly unstable but Bourgain and Wang \cite{BW} observed that specific perturbations of the initial data   preserve blow up with speed $1/t$. The excess of mass in this regime converts into a {\it flat and smooth} asymptotic profile at the blow up time, which does not alter the blow up law.

\begin{theorem}[Bourgain-Wang solutions \cite{BW}]\label{BW}
Let $d=1,2$. Let $u^*$ be such that
\begin{equation}\label{smoothnessa}
u^*\in X_A=\{f\in H^A \hbox{ with } (1+|x|^A) f \in L^2\},
\end{equation}
and
\begin{equation}\label{flatnessa}
D^{\alpha} u^*(0)=0,\hbox{ for } 1\leq |\alpha| \leq A,\end{equation}
for some $A$ large enough.
 Then, there exists a solution $u_{BW}\in \mathcal C((-\infty,0),H^1)$ to (\ref{nls}) which blows up at $t=0$, $x=0$ and satisfies:
\be
\label{blowupbourgaion}
u_{BW}(t)-S_{\rm NLS}(t)\to u^* \ \ \mbox{in} \ \ H^1 \ \ \mbox{as} \ \ t\uparrow 0.
\ee
\end{theorem}

We refer to \cite{BW} for a more precise statement and  to \cite{KSNLS} for a further discussion on the   construction of a manifold of Bourgain-Wang solutions. Note that the Bourgain-Wang blow up  solutions   saturate the lower bound \eqref{lowerbound}: $$\|\nabla u(t)\|_{L^2}\sim \frac{1}{T-t}.$$

Recall that by Strichartz estimates and   $L^2$ critical Cauchy theory \cite{CW},   solutions which scatter are $L^2$ stable. Also,  from Theorem \ref{theoremdeux},  solutions in $\mathcal{B}_{\alpha^*}$ that blow up in finite time {in the log-log regime} form an open set in $H^1$. Bourgain-Wang solutions correspond in a certain sense to a threshold unstable dynamic between these two stable dynamics, as  was proved in \cite{MRS}.

\begin{theorem}[Instability of Bourgain-Wang solutions, \cite{MRS}]
\label{theoremmainbw}
Let $d=2$. Let $u^*$ satisfy \eqref{smoothnessa} and \eqref{flatnessa} and let $u_{BW}\in \mathcal C((-\infty,0),H^1)$ be the corresponding Bourgain-Wang solution. Then there exists a continuous map $$\eta\in [-1,1]\to u^\eta(-1)\in \Sigma=\{f\in H^1(\R^d) \hbox{ with } xf\in L^2(\R^d)\}$$ such that, $u^{\eta}(t)$ being the solution of \eqref{nls} with initial data   $u^{\eta}(-1)$ at $t=-1$, 
\begin{itemize}
\item   $u^{\eta=0}(t)\equiv u_{BW}(t)$;
\item $\forall \eta\in(0,1]$, $u^{\eta}\in \mathcal C(\Bbb R,\Sigma)$ is global in time and scatters;
\item $\forall \eta\in [-1,0)$, $u^{\eta}\in \mathcal C((-\infty,T^{\eta}),\Sigma)$ blows up   in the log-log regime at $-1<T^{\eta}<0$.
\end{itemize}
\end{theorem}

Note that this theorem only describes the flow near the Bourgain-Wang solution along {\it one} instability direction. A main open problem for (NLS)  is to completely describe the flow near the ground state $Q$. Theorem \ref{theoremmainbw} is a first step towards the description of the flow near   Bourgain-Wang solutions, which is an interesting   open problem (see section~\ref{sec:openproblem} and \cite{KSNLS}).


\subsection{Structural instability of the log-log law}


The Zakharov system in space dimensions $d=2,3$, is
a model related to (NLS) with fundamental physical relevance \cite{SS}:
\be
\label{zakharov}{\rm (Zakharov)} \qquad 
\left \{ \begin{array}{ll}
         iu_t=-\Delta u+nu,\\
         \frac{1}{c_0^2}n_{tt}=\Delta n +\Delta |u|^2,
         \end{array}\qquad \qquad
\right .
\ee
for some fixed constant $0<c_0<+\infty$. In the limit $c_0\to +\infty$, one formally recovers (NLS). In dimension $d=2$, this system displays a variational structure like the one of (NLS), even though the scaling symmetry is destroyed by the wave coupling. By bifurcation from $S_{\rm NLS}(t)$, the exact  solution of (NLS), Glangetas, Merle \cite{GM} constructed a one parameter family of blow up solutions with blow up speed: $$\|\nabla u(t)\|_{L^2}\sim \frac{C(u_0)}{T-t}.$$ 
Numerical experiments (Papanicolaou, Sulem, Sulem, Wang, \cite{PSSW}) suggest that such solutions are stable. Recall also that from Merle \cite{MZakh}, {\it all finite time blow up solutions to (\ref{zakharov}) satisfy the lower bound} $$\|\nabla u(t)\|_{L^2}\geq \frac{C(u_0)}{T-t}.$$ In  particular, there are no log-log blow up solutions for (\ref{zakharov}), which means that in some sense    Zakharov blow up dynamic should be   more stable    than its asymptotic limit (NLS). This is one more hint that the stable log-log law for (NLS)  deeply relies  on some specific algebraic structure of (NLS), as suggested by the use of non linear degeneracy properties in the blow up analysis   of Theorem \ref{theoremdeux}. We emphasize that refined study of the singularity formation for the Zakharov system is mostly open, though it can be considered as the first step towards the understanding of   relevant and more complicated systems related to Maxwell  equations.


\subsection{Blow up for the  $L^2$ supercritical (NLS) model}\label{sec:l2superctritical}


Consider   the $L^2$ supercritical (NLS) model
\be
\label{nlsl2super}
{\rm (NLS)} \ \ \left\{\begin{array}{ll} i\pa_t u +\Delta u +u|u|^{p-1}=0\\ u_{|t=0}=u_0\end{array}\right . \ \ (t,x)\in \Bbb R\times \Bbb R^d,
\ee
with the choice
$$p>1+\frac{4}{d}.$$
The scaling symmetry for \eqref{nlsl2super} is given by
$$u_{\l}(t,x)=\l^{\frac{2}{p-1}}u(\l^2t,\l x),$$
and the homogeneous Sobolev space invariant under this symmetry is $\dot{H}^{s_c}(\mathbb{R}^d)$ where 
$$s_c=\frac{d}{2}-\frac{2}{p-1},$$
and $s_c>0$ corresponds to the $L^2$ supercritical case. The problem is said energy subcritical if $s_c<1$, energy critical if $s_c=1$, and energy supercritical if $s_c>1$.
We  now briefly describe   three results for the $L^2$ supercritical NLS obtained by extending  some techniques developed for the results reviewed above, thus illustrating their robustness. 

\subsubsection{Standing ring solutions} 

We consider in this section the quintic NLS, i.e $p=5$ in \eqref{nlsl2super}, which is $L^2$ critical in dimension $d=1$ and $L^2$ supercritical in dimensions $d\geq 2$. Radially symmetric solutions blowing up on a (nontrivial) sphere  are constructed in \cite{R2} for dimension $d=2$, and in \cite{RScmp} for the  dimensions $d\geq 3$. Note that this covers energy subcritical ($d=2$), energy critical ($d=3$) and energy supercritical ($d\geq 4$) cases. Heuristically, in radial variables, the problem becomes: $$i\partial_t u+\partial_r^2u+\frac{N-1}{r}\partial_ru+u|u|^4=0;$$  one   expects that if the singularity formation occurs on the unit sphere, then close to the singularity $$\left|\frac{\partial_ru}{r}\right|\sim |\partial_r u|<<\left|\partial_r^2u\right|,$$ so that the leading order blow up dynamics should be given by the one dimensional quintic NLS. Therefore,   solutions blowing up on a sphere in any dimension $d\geq 2$ can be constructed  by perturbation of   the log-log dynamic for the $L^2$ critical one dimensional (NLS).

See further extensions  in \cite{HR1}, \cite{HR2}, \cite{Z}, where axially symmetric blow up  solutions are constructed     for the cubic NLS.
\subsubsection{Self similar solutions}

In the $L^2$ supercritical case, the so-called self similar regime,  i.e.  blow up of the type
\begin{equation}\label{taux}
\|\nabla u(t)\|_{L^2}\sim \frac{1}{(T-t)^{\frac{1}{p-1}+\frac{1}{2}-\frac{d}{4}}}\quad \textrm{as }t\sim T,
\end{equation}
is  conjectured to be the stable blow up regime, in particular  in view of   numerical computations (see for example \cite{SS}). In \cite{MRSgafa},  this conjecture is actually proved in the ``slightly'' $L^2$ supercritical case. More precisely,  considering the focusing nonlinear Schr\"odinger equations \eqref{nlsl2super} with $p>1+\frac{4}{d}$  sufficiently close to $1+\frac{4}{d}$,     existence and stability in $H^1$ of  self similar finite time blow up dynamics is proved together with a qualitative description of the singularity formation near the blow up time. The analysis is perturbative of the log-log analysis of the $L^2$ critical case reviewed in section \ref{sec:loglog}. 

\subsubsection{Collapsing ring solutions}\label{sec:ringblowup}
For (NLS) equation \eqref{nlsl2super} with $d\geq 2$ and $p$ in the range 
$$1+\frac 4d<p<\min\left(\frac{d+2}{d-2},5\right),$$
it is proved in  \cite{MRSring}  that
any radially symmetric solution $u(t)$ blowing up at time $T$ satisfies the following universal upper bound
\be
\label{esiheoeogh}
\int_t^T(T-\tau)\|\nabla u(\tau)\|_{L^2}^2d\tau\leq C(u_0)(T-t)^{\frac{2\alpha}{1+\alpha}},
\ee
where $\alpha$ is given by
$$\alpha=\frac{5-p}{(p-1)(N-1)}.$$ 
The upper bound \eqref{esiheoeogh} is proved to be sharp since it is actually achieved by a family of collapsing ring blow up solutions whose existence was first formally predicted in \cite{FG}. The construction of   ring solutions in  \cite{MRSring} relies on a robust strategy to build {\it minimal blow up elements} developed in \cite{RS2010}. In particular, it is a {\it non dispersive} solution, i.e. a solution that concentrates all its $L^2$ mass at the origin as $t$ converges to the blow-up time.


\subsection{Energy critical problems}\label{sec:energy}


{Energy critical problems} have also recently attracted  considerable attention, in particular wave maps, Schr\"odinger maps and the harmonic heat flow: 
$$
 \begin{array}{lll}\mbox{(Wave Map)}\ \ \pa_{tt}u-\Delta u=(|\nabla u|^2-|\pa_tu|^2)u\\
\mbox{(Schr\"odinger map)}\ \ u\wedge \pa_tu=\Delta u+|\nabla u|^2u\\
\mbox{(Harmonic Heat Flow)}\ \ \pa_tu= \Delta u+|\nabla u|^2u
\end{array}
 \qquad  (t,x)=\Bbb R\times \Bbb R^2, \ u(t,x)\in \Bbb S^2.
$$
These problems exhibit ground state stationary solutions which are minimizers of the associated energy. This energy is also left invariant by the scaling symmetry of the problem. The long standing question of existence of blow up solutions  is solved in \cite{KST}, \cite{RodSter}, \cite{RR2009} for the wave map problem and in \cite{MRR} for the Schr\"odinger map problem. New types of dynamics, including ground state-like behavior or infinite time blow up, are obtained in \cite{KNS}, \cite{DK} for the related semi linear wave equation. For the harmonic heat flow, the full sequence of stable and unstable blow up regimes is obtained in the two papers \cite{Rstud,Rstud2}, which shades some new light on the structure of the flow near the ground state. These works provide a better understanding of the blow up scenario by the determination of explicit regimes, but the full description of the flow, even for initial data close to the ground state, is far from being complete. Another series of works \cite{KDM1}, \cite{KDM2}, aims at classifying all possible behaviors of  the solutions, in particular solving the so-called ``soliton resolution conjecture'' for the energy critical focusing non linear wave equation in   dimension $3$.

\subsection{Open problems}\label{sec:openproblem}

Here is a (non exhaustive) list of interesting open questions concerning the description of the flow near the ground state solitary wave for the $L^2$ critical (NLS). 
\begin{enumerate}
\item Does the set of initial data corresponding to Bourgain-Wang solutions form a codimension 1 manifold? This is clearly suggested by \cite{MRS} and \cite{KSNLS}. 

\item    In \cite{Rannalen}, it is shown that for initial data close to the ground state in the energy space, a blow up solution either follows the log-log regime, or blows up at least as fast as the conformal blow up speed (see section \ref{sec:loglog}). Do there exist blow up solutions in $H^1$ such that the corresponding blow up speed is strictly faster than the  conformal one? Another related question in this context concerns the   existence of a general upper bound on the blow-up speed.

\item Can one obtain a complete classification of the flow near the solitary wave in the spirit of the recent results  for $L^2$ critical gKdV in \cite{MMR1}, \cite{MMR2}, \cite{MMR3} (see section \ref{sectionkdv})?
\end{enumerate}


\section{The $L^2$ critical (gKdV) problem}\label{sectionkdv}


We now summarize the state of the art for the $L^2$ critical (gKdV) problem \eqref{kdv}, and in particular recent results obtained in \cite{MMR1}, \cite{MMR2}, \cite{MMR3}. The gKdV model is one dimensional and we consider $Q$ the unique (up to translation) $H^1$  explicit solution of 
$Q''+Q^5= Q$, i.e. $Q(x) = 3^{\frac 14} \cosh^{-\frac 12}(2x)$.


\subsection{Existence of blow up solutions}


As discussed in the Introduction, the local $H^1$ Cauchy theory of Kenig, Ponce and Vega \cite{KPV} combined with Weinstein's variational characterization of the ground state \cite{W1983}, ensure   that $H^1$ initial data with subcritical mass $\|u_0\|_{L^2}<\|Q\|_{L^2}$ yield global in time $H^1$ solutions. Scattering is   known only for small (no control) $L^2$ data \cite{KPV2}. One  motivation to address  this model lies on its similarities with (NLS) but without the   specific structure of the (NLS) problem, in particular the pseudo conformal symmetry.
Note that recent progresses on critical (NLS), (gKdV) and other nonlinear dispersive models have be obtained in synergy, any new result or technique on one model being source of inspiration for the others.
\\

The study of singularity formation for $H^1$ initial data 
with mass close to the minimal mass
\begin{equation}
\label{utwosmall}
\|Q\|_{L^2}\leq\|u_0\|_{L^2}<\|Q\|_{L^2}+\alpha^* \quad \hbox{for}\quad \alpha^*\ll1, 
\end{equation} 
has been initiated in the series of works \cite{MMjmpa,MMgafa,Mjams,MMannals, MMduke, MMjams}. As for (NLS), variational constraints imply that under \fref{utwosmall}, if the solution blows up in finite time $T<+\infty$, then it admits near the blow up time a decomposition of the form
\be\label{decompkdv}
u(t,x)=\frac{1}{\l^{\frac 12}(t)}(Q+\e)\left(t,\frac{x-x(t)}{\l(t)}\right) \quad \hbox{where} \quad \|\e(t)\|_{H^1}\leq \delta(\alpha^*),
\ee
where $\lim_{\alpha^*\to 0} \delta(\alpha^*)=0$. Therefore, as for (NLS),  the problem reduces to undertand the coupling between the finite dimensional   dynamics governing the solitary wave part, here $(\l(t),x(t))$, and the residual part $\e(t)$. Two new tools were introduced 
in \cite{MMjmpa,Mjams,MMannals,MMjams} for (gKdV), which later were extended to   (NLS):
\begin{itemize}
\item monotonicity formula and  localized virial identities on $\varepsilon$;
\item Liouville type theorems to classify asymptotic solutions.
\end{itemize}
The classical conjecture of existence of blow up dynamics was solved in \cite{Mjams, MMjams}. The original proof in \cite{Mjams} is indirect and based on a classification argument \cite{MMjmpa}. Basic information on the blow up structure was obtained  in \cite{MMannals}, in particular the asymptotic stability of $Q$ as a blow up profile, i.e. $$\e(t)\to 0\ \mbox{in}\ \ L^2_{\rm loc}\ \ \mbox{as}\ \ t\to T.$$ 
Under the further assumption
\be\label{classe2}
\int_{x'>x} u_0^2(x') dx' < \frac {C}{x^6} \ \ \mbox{for}\ \ x>0,
\ee
the first   result of finite time blow up for negative energy solutions satisfying \eqref{smallamss} is proved in \cite{MMjams}, together with a   partial result concerning the blow up rate:
$$\|\nabla u(t_n)\|_{L^2}\lesssim \frac{C_u}{T-t_n}, \quad \hbox{for subsequence $t_n\to T$.}$$
We will see below that a decay assumption of the type \eqref{classe2} is indeed necessary for finite time blow up.  Finally, it is shown in  \cite{MMduke} that $H^1$ solutions with minimal mass $\|u_0\|_{L^2}=\|Q\|_{L^2}$ and decay to the right \fref{classe2}, are global in time ($t\geq 0$), which rules out the possibility of minimal mass blow up under assumption \fref{classe2}.
 

\subsection{Blow  up description and classification of the flow  near the ground state}


We now summarize a series of more recent works \cite{MMR1}, \cite{MMR2}, \cite{MMR3} which give a complete description of the flow near the ground state, thus completing \cite{MMjmpa,MMgafa,Mjams,MMannals,MMjams,MMduke}. With respect to these earlier works, improved results involve techniques developed for the study of (NLS) \cite{MRannals,MRgafa,MRinvent,MRjams,MRcmp,Rannalen} and energy critical geometrical problems \cite{RR2009,MRR,Rstud}.
\\ 

Consider the following set of initial data
$$
\mathcal{A}=
\left\{
u_0=Q+\e_0   \hbox{ with } \|\e_0\|_{H^1}<\alpha_0 \hbox{ and }
\int_{y>0} y^{10}\e_0^2< 1
\right\} \hbox{ for $\alpha_0>0$ small}.$$
To consider  initial data with decay on the right, instead of data simply in the energy space, is necessary for the (gKdV) equation -- see
Theorem \ref{th:4} below. Despite   many analogies between the two problems, it is  a fundamental difference with respect to   (NLS) analysis.

\begin{theorem}[Blow up for nonpositive energy solutions in $\mathcal A$, \cite{MMR1}]\label{th:1} 
Let $0<\alpha_0\ll 1$. Let $u_0 \in \mathcal{A}$.
If $E(u_0)\leq 0 $ and  $u_0$ is not a soliton, then $u(t)$ blows up in finite time
$T$ and  there exists $ \ell_0=\ell_0(u_0)>0$ such that
\be\label{blow}
\|u_x(t)\|_{L^2} = \frac {\|Q'\|_{L^2}+o(1)}{\ell_0 (T-t)} \quad 
\hbox{as $t\to T$.}
\ee
Moreover,
there exist $\lambda(t)$, $x(t)$ and $u^*\in H^1$, $u^*\neq 0$,
such that 
\begin{equation}\label{th1.1}
u(t,x)-\frac{1}{\lambda^{\frac12}(t)}Q\left(\frac{x-x(t)}{\lambda(t)}\right)\to u^*\ \ \mbox{in}\ \ L^2\ \ \mbox{as}\ \ t\to T,
\end{equation} 
\be\label{bbg}\lambda(t)= (\ell_0+o(1)) (T-t), \quad x(t)= \left(\frac 1 {\ell_0^2}+o(1)\right)\frac 1{(T-t)}\ \ \mbox{as} \ \ t\to T,\ee
\end{theorem}

\noindent{\bf Comments on Theorem \ref{th:1}}

\medskip

\noindent{\it 1. Blow up speed and stable blow up.}  An important feature of Theorem \ref{th:1} is the derivation of the   blow up speed for $u_0\in \mathcal{A}$ with non positive energy: \be\label{bll}\|u_x(t)\|_{L^2}\sim \frac{C(u_0)}{T-t}\ee 
which implies in particular that $x(t)\to +\infty$ as $t\to T$.
The concentrating soliton and the remainder term $u^*$ thus split spatially.
Observe that the blow up speed is   far above the scaling blow up law which would be for (gKdV): $\|u_x(t)\|_{L^2}\sim c (T-t)^{-\frac13}$
(see  \cite{MRR}, \cite{Rstud} for a similar gap phenomenon in energy critical geometrical problems ). 
\\
To complement Theorem 1, one knows that the set of initial data in $\mathcal A$ which led to the blow up \eqref{blow}--\eqref{bbg} is open in the $H^1$ topology (see \cite{MMR1}). Thus, $(T-t)^{-1}$ is the  {\it stable blow up}  behavior
for (gKdV), which is in contrast with (NLS).

\medskip

\noindent{\it 2.   Decay assumption on the right.} Let us stress the importance of the decay  assumption on the right in space for the initial data, which was already fundamental in the earlier works \cite{MMjams}, \cite{MMduke}.
Indeed, in contrast with the NLS equation, the universal dynamics can not be seen in $H^1$ and an additional assumption of decay to the right  is required (see Theorem \ref{th:4} below).
Note however, that  we do not claim sharpness in the weight $y^{10}$ in Theorem \ref{th:1}.

\medskip

\noindent{\it 3.   Dynamical characterization  of $Q$}: Recall from the variational characterization of $Q$ that $E(u_0)\leq 0$ implies $\|u_0\|_{L^2}>\|Q\|_{L^2}$, unless $u_0\equiv Q$ up to scaling and translation symmetries. Theorem \ref{th:1} therefore recovers  the dynamical classification of $Q$ as the unique global zero energy solution in $\mathcal{A}$, like for the mass critical (NLS), see \cite{MRjams}. 

\medskip

A key to the complete description of the flow around $Q$ is  the minimal mass case,   solved by the following result.
 
\begin{theorem}[Existence and uniqueness of the minimal mass blow up element, \cite{MMR2}]
\label{th:2}\quad \\
{\rm (i) Existence.} There exists a   solution $S_{\rm KdV}(t)\in \mathcal C((0,+\infty),H^1)$ of \fref{kdv} with minimal mass $$\|S_{\rm KdV}(t)\|_{L^2}=\|Q\|_{L^2}$$ which blows up backward at the origin:
\be
\label{speedlvowup}
\|S_{\rm KdV}(t)\|_{H^1}= \frac {\|Q'\|_{L^2}+o(1)}t \ \ \mbox{as} \ \ t\downarrow 0,
\ee
$$S_{\rm KdV}(t,x)-\frac{1}{t^{\frac 12}}Q\left(\frac{x+ \frac 1t+\bar{c}t}{t}\right)\to 0\ \ \mbox{in}\ \ L^2\ \ \mbox{as}\ \ t\downarrow 0$$ 
where $\bar{c}$ is a universal constant.

\noindent{\rm  (ii) Uniqueness. } Let $u_0\in H^1$ with $\|u_0\|_{L^2}=\|Q\|_{L^2}$ and assume that the corresponding solution $u(t)$ to \fref{kdv} blows up in finite time. Then $$u\equiv S_{\rm KdV}$$ up to the symmetries of the flow.
  \end{theorem}
  
For critical (gKdV), one recovers both the existence of a minimal mass solution, completely analoguous to $S_{\rm NLS}(t)$, and the analogue of Merle's classification result for (NLS), i.e. Theorem~\ref{critmass}.
In the absence of pseudo conformal symmetry,   the uniqueness part of the proof of Theorem \ref{th:2} is   completely dynamical and closely related to the analysis of the inhomogeneous NLS model in \cite{RS2010}.
The existence part is related to the  universality of ``the exit regime'', see Theorem \ref{th:3} below and   Section 4 for a sketch of proof.

Observe that $S_{\rm KdV}(t)$ blows up  with  the same   speed as the stable blow of Theorem~\ref{th:1}. However,  the minimal mass blow up is by essence unstable by perturbation since for example initial  data $S_\e(-1)=(1-\e)S_{\rm KdV}(-1)$, for $0<\e<1$,  have subcritical mass and   thus lead to global and bounded solutions.

\medskip

We now recall the main result from \cite{MMR1, MMR2}, which classifies the flow for initial data in $\mathcal A$. Define the $L^2$ modulated tube around the soliton manifold:
\be\label{tube}
\mathcal T_{\alpha^*}=\left\{u\in H^1\ \ \mbox{with}\ \ \inf_{\l_0>0, \ x_0\in \RR}
\Big\|u-\frac 1{\lambda_0^{\frac 12}} Q\left(\frac {.-x_0}{\lambda_0} \right) \Big\|_{L^2} <\alpha^*\right\}.\ee 
Here $\alpha_0,\alpha^*$ are universal constants related as follows 
\be
\label{relationalpha}
0<\alpha_0\ll \alpha^*\ll1.
\ee

\begin{theorem}[Rigidity of the dynamics in $\mathcal{A}$, \cite{MMR1,MMR2}]
\label{th:3}
For $u_0 \in \mathcal{A}$, only three scenarios are possible:

 \medskip

\noindent{\em (Blow up)}  For all $t\in [0,T),$ $u(t)\in \mathcal{T}_{\alpha^*}$ and the solution blows up in finite time $T<+\infty$ in the regime described in  Theorem \ref{th:1}: \eqref{blow}, \eqref{th1.1} and \eqref{bbg}.
\medskip

\noindent{\em (Soliton)}  The solution is global, for all $t\geq 0, $ $u(t)\in \mathcal{T}_{\alpha^*}$, and  there exist $\lambda_{\infty}>0$ and $x(t)$ such that  
\begin{equation}
\l_\infty^{\frac 12} u(t,\l_\infty \cdot +x(t))\to Q\quad \hbox{in $H^1_{\rm loc}$ as $t\to +\infty$},
\end{equation} 
\be
\label{nveonveonoene}
|\l_{\infty}-1|\leq o_{\alpha_0\to 0}(1),\quad 
 x(t)\sim \frac{t}{\lambda_{\infty}^2} \ \ \mbox{as}\ \ t\to +\infty.
\ee 
\medskip

\noindent{\rm (Exit)}  There exists $t^*\in (0,T)$ such that 
$u(t^*)\not \in \mathcal{T}_{\alpha^*}$. Let $t^*_{u}\gg 1$ be the corresponding exit time
\be\label{exittime}
t^*_{u}=\sup\{ 0<t<T;\ \hbox{such that } \forall t'\in [0,t], \ u(t)\in \mathcal T_{\alpha^*}\}.
\ee
Then there exist  $\tau^*=\tau^*(\alpha^*)$ (independent of $u$) and $(\lambda_u^*,x_u^*)$ such that 
$$\left\|  (\l_u^*)^{\frac 12} u\left(t_u^*,   \l_u^* x +  x_u^*\right) -S_{\rm KdV}(\tau^*,x)\right\|_{L^2} \leq \delta(\alpha_0),$$
where 
$\delta(\alpha_0) \to 0$ as ${\alpha_0\to 0}.$
\end{theorem}
 
 Theorem \ref{th:3} classifies the  possible behaviors  of solutions with initial data in $\mathcal A$, including a certain description of the long time dynamics in the (Exit) regime. This question happens to be deeply related to the  {minimal mass dynamics}, which is an unexpected new phenomenon. Indeed, any (Exit) solution at its exit time is close to a universal profile $S_{\rm KdV}(\tau^*)$, up to a large defocusing scaling $\l^*_u\gg 1$,   premise of a dispersive behavior.
In view of the universality of $S_{\rm KdV}$ as a attractor to all exiting solutions, and in continuation of Theorem \ref{th:3}, it is thus an important open problem 
  to understand the behavior of $S_{\rm KdV}(t)$ as $t\to +\infty$. For the mass critical (NLS), $S_{\rm NLS}(t)$  scatters  as $t\to\infty$. Scattering of $S_{\rm KdV}(t)$ as $t\to +\infty$ is an open problem\footnote{by scattering for (gKdV), we mean that there exists a solution $v(t,x)$ to the Airy equation $\pa_tv+v_{xxx}=0$ such that $\lim_{t\to +\infty}\|S_{\rm KdV}(t)-v(t)\|_{L^2}=0$.}. We conjecture that $S_{\rm KdV}(t)$ actually scatters, and because scattering is an open property in $L^2$  (see \cite{KPV2}), we obtain the following corollary.
  
\begin{corollary}[\cite{MMR2}]
Assume  that $S_{\rm KdV}(t)$ scatters as $t\to+\infty$. Then, any solution in the {\em (Exit)} scenario is global for positive time and  scatters as $t\to +\infty$.
\end{corollary}

Related rigidity theorems near   solitary waves were also obtained by Nakanishi and Schlag \cite{NS1,NS2} and Krieger, Nakanishi and Schlag \cite{KNS}, for super critical wave and Schr\"odinger equations using the invariant set methods of Beresticky and Cazenave \cite{BeresCaze},   Kenig-Merle concentration compactness approach \cite{KM}, the classification of minimal dynamics \cite{DM2,DM}, and a further ``no return'' lemma in the (Exit) regime.
  In  the (Exit) regime, this lemma shows that the solution cannot come back close to solitons  and   in fact scatters. (In   critical situations, such an analysis is more delicate and incomplete, see \cite{KNS}.) Both the blow up statements and the no return lemma in \cite{NS1,NS2}, rely on a specific algebraic structure - the virial identity - which does not exist for (gKdV). 
The above results for (gKdV) rely on the {explicit} computation of the solution in the various regimes, and not on algebraic virial type identities. Indeed introducing the decomposition \eqref{decompkdv}, one shows  that at the leading order, $\l(t)$ satisfies  
\be
\label{vnoneoneo}
\l_{tt}=0, \ \ \l(0)=1.
\ee Roughly speaking, the three regimes (Exit), (Blow up), (Soliton)   correspond respectively to $\l_t(0)>0$, $\l_t(0)<0$ and the threshold dynamic $\l_t(0)=0$, see next section for details.\\ 

\medskip   
   
Now we produce a wide range of different blow up rates, including grow up in infinite time, for initial data $u_0\not \in \mathcal{A}$ having slow decay on the right. In particular,  the  blow up rate $\frac 1{(T-t)}$   is universal in $\mathcal A$ as a consequence of the leading order ODE \eqref{vnoneoneo} which can be justified only under the decay assumption $u_0\in \mathcal A$. The tail of slowly decaying data can interact with the solitary wave which   is moving to the right, which   lead in some cases to new exotic singular regimes.

\begin{theorem}[Exotic blow up rates, \cite{MMR3}]\label{th:4}\quad \\
{\rm (i)  Blow up in finite time:} for any $\nu > \frac {11}{13}$,  there exists $u \in  C((0,T_0),H^1)$   solution  of \eqref{kdv} blowing up at $t=0$ with
\be\label{th:4:1}
\|u_x(t)\|_{L^2} \sim t^{-\nu}\ \ \mbox{as}\ \ t\to 0^+.\ee
{\rm (ii)  Grow up in infinite time:} there exists  $u\in  C([T_0,+\infty),H^1)$ solution of \eqref{kdv} growing  up at $ +\infty$ with
\be\label{th:4:4}
\|u_x(t)\|_{L^2} \sim e^{t} \ \ \mbox{as}\ \ t\to +\infty.
\ee
For any $\nu>0$,  
there exists  $u\in  C([T_0,+\infty),H^1)$ solution of \eqref{kdv} growing  up at $ +\infty$ with
\be\label{th:4:5}
\|u_x(t)\|_{L^2} \sim t^{\nu}\ \ \mbox{as}\ \ t\to +\infty.
\ee
Moreover, such solutions can be taken arbitrarily close in $H^1$ to  solitons.\end{theorem}

It follows from the proof of Theorem \ref{th:4} that the grow up rate is directly related to the peculiar behavior of the initial data on the right. In particular, other types of blow up speeds can be produced by changing the tail of the initial data. A   similar phenomenon was observed for global in time growing up solutions to the   energy critical harmonic heat flow by Gustafson, Nakanishi and Tsai \cite{NT2}. In \cite{NT2}, an explicit formula on the growth of the solution at infinity is given directly in terms of the initial data. Continuums of blow up rates were also observed in   pioneering works by Krieger, Schlag and Tataru \cite{KST,KST2} for   energy critical wave problems, see also Donninger and Krieger \cite{DK}.  We also refer to   stability    \cite{KT} and finite time blow up \cite{MRR} results for the energy critical Schr\"odinger map problem. All these results point out that the critical topology is not enough by itself to classify the flow near the ground state.


\section{An overview of the classification of the flow for (gKdV)}\label{sec:kdvproof}


Our aim in this section is to give a simple overview of the proof of the classification theorems for the (gKdV) flow, and in particular the rigidity property formally stated in \fref{vnoneoneo} which readily leads to the three scenarios in Theorem \ref{th:3}.\\

 \noindent{\bf Notation.} Let 
\be
\label{defL}
Lf=-f''+f-5Q^4f
\ee
be the linearized operator close to $Q$. We note the rescaling  operator:
$$f^{ \l }=\frac 1 {\lambda^{\frac 12}}  f \left( \frac{x }{\lambda}\right)$$ and   introduce the generator of $L^2$ scaling:
$$\Lambda f=-\frac{\pa f^{\l }}{\pa\lambda}_{|\l=1}=\frac12f+yf'.$$ The $L^2$ scalar product is denoted by $(f,g)=\int fgdx.$


\subsection{The approximate blow up profile}


The starting point of the analysis is to understand the leading order profile of the solution. The knowledge of the instability directions of the flow near $Q$ is essential and   adjustments of these directions lead to the scenarios of Theorem~\ref{th:3}.\\
Let us look for a solution to (gKdV) in the form of a renormalized bubble
\be\label{ans}u(t,x) =\frac 1 {\lambda^{\frac 12}(t)}  v  \left(s, \frac{x-x(t)}{\lambda(t)}\right),\ \ \frac{ds}{dt}=\frac1{\l^3}
\ee
which leads to the evolution equation: 
$$\pa_sv+(v_{yy}-v+v^5)_y+\left(\xsl-1\right)v_y-\lsl \Lambda v=0.$$ We freeze the modulation equations 
\be
\label{eqhiehefof}
\xsl-1=0, \ \ -\lsl=b
\ee
and we follow the strategy of construction of blow up solutions initiated in \cite{MRgafa}, \cite{RR2009}, \cite{RS2010}, looking for a slowly modulated profile $$v(s,y)=Q_{b(s)}(y), \ \ Q_b=Q+b P_1+b^2 P_2+\dots.$$ The problem    writes: how can we choose the law of the   parameter $b$ to ensure the solvability of the   system satisfied by $(P_i)_{i\ge 1}$? We   proceed the expansion $$b_s=-c_2b^2-c_3b^3+\dots, \ \ Q_b=Q+bP_1+b^2P_2+\dots, \ \ v(s,y)=Q_{b(s)}(y)$$ and aim at solving the following system by expanding in powers of $b$ 
\begin{align*}
0& = \pa_sv+(v_{yy}-v+v^5)_y+\left(\xsl-1\right)v_y-\lsl \Lambda v\\
&=b_s\frac{\partial Q_b}{\partial b}+((Q_b)_{yy}-Q_b+Q_b^5)_y+b\Lambda Q_b.
\end{align*}
\begin{itemize} 
\item Order $O(1)$  corresponds to the solitary wave equation: $$ (Q_{yy}-Q+Q^5)_y=0.$$
\item Order $O(b)$ provides the   equation 
\be
\label{ncieoneonoev} \Lambda Q+(LP_1)'=0
\ee where $L$ is  given by \fref{defL}. The translation invariance induces the existence of a (unique by ODE techniques) non trivial element of the kernel: $L Q'=0$. Hence   solvability of \fref{ncieoneonoev}   requires the cancellation $$(\Lambda Q,Q)=0$$ which indeed holds as a direct consequence of the $L^2$ critical invariance of the problem. However since $\int_{-\infty}^{+\infty} \Lambda Q\neq 0$, solving \fref{ncieoneonoev}  induces a non trivial growth of the solution $P_1$ to the right or to the left. Choose $P_1$   the unique solution to \fref{ncieoneonoev} which is exponentially well localized on the right $y\to +\infty$, but behaves like a non zero constant as $y\to -\infty$.
\item Order $O(b^2)$: we obtain an   equation of the form $$-c_2P_1+(LP_2+N(P_1))'=0$$ where $N(P_1)$ is some explicit nonlinear term. The solvability condition $$(-c_2P_1+(N(P_1))', Q)=0$$ leads after some computations to the choice $$c_2=2.$$
\end{itemize}
Therefore,  at the formal level, we obtain the following {\it universal} dynamical system driving the geometrical parameters: 
\be\label{simpl}
b_s=-2b^2, \ \ b=-\lsl, \ \ \xsl=1, \ \ \frac{ds}{dt}=\frac 1{\l^3},
\ee which   is equivalent to: 
\be
\label{systenqui}
\l_{tt}=0,\ \ x_t=\frac{1}{\l^2}, \ \ b=-\l^2\l_t.
\ee
By   scaling invariance, we may always choose $\lambda(0)=1$, and then the phase portrait of \fref{systenqui} is
\begin{itemize}
\item if $b(0):=b_0<0$, then $\lambda(t)=1-b_0t\to +\infty$ as $t\to +\infty$.
\item if $b_0=0$, then $\lambda(t)=1$ for all $t\geq 0$.
\item if $b_0>0$, then $\lambda (t)=1-b_0t$ vanishes at $T=\frac{1}{b_0}$ and   $\l(t)=b_0(T-t)$,   integrating from $T$.
\end{itemize}
The analysis is  now reduced to show that for initial data in $\mathcal A$, the above dynamical system governs the leading order solitonic part of the solution, so that the three  regimes of Theorem~\ref{th:3}  correspond to perturbations of the above three cases. Like for the finite dimensional system \fref{systenqui}, the first and third  regimes are stable, the second one is the unstable threshold dynamics.


\subsection{Nonlinear decomposition of the flow}


A solution $u(t,x)$ of \eqref{kdv} close in $H^1$ to a soliton is   decomposed as
$$
u(t,x) =\frac 1 {\lambda^{\frac 12}(t)} (Q_{b(t)} +\e) \left(s,y \right), \ \ y=\frac{x-x(t)}{\lambda(t)}, \ \ \frac{ds}{dt}=\frac{1}{\l^3(t)}$$ where $(s,y)$ are the renormalized space time variables. The parameters $(b(t),\l(t),x(t))$ which describe the drift along the soliton family are uniquely adjusted to obtain orthogonality conditions on $\e$ for all time:
\be\label{ortho}
( \e, Q)=(\e, \Lambda Q )=( \e, y \Lambda Q)=0.
\ee
This choice of orhogonality conditions is justified by Lemma \ref{tout} below.
The renormalized flow for $\e$ becomes:
\be\label{eqeps}
\e_s - (L \e)_y  = \left(  \lsl + b\right) \Lambda Q + \left(\xsl-1\right) Q' + \lsl \Lambda \e
+ O\left(b^2 + |b_s| + |\e|^2\right).
\ee
The dynamical system driving $(b,\l,x)$ is obtained from this equation and \eqref{ortho}. With respect to the idealized system \eqref{simpl}, it contains additional   perturbation terms coming from 
$b^2$, $b_s$ and $\e$. Typically, the construction of $Q_b$ and the decay assumption on the initial data (it is necessary at this point) ensure a bound of the form: \be
\label{contorlloca}
|b_s+2b^2|+\left|\lsl+b\right|\lesssim \int |\e|^2e^{-|y|}+O(b^3).
\ee
The remaining point to justify the dynamics is to obtain a uniform control of $\e(s)$ in   suitable norms whose choice  is essential.\\

Neglecting for the moment second order and higher order terms in the equation of $\e$, we concentrate on a toy model:
\be\label{eq:toy}
\widetilde \e_s - (L\widetilde \e)_y = \alpha(s) \Lambda Q + \beta(s) Q'
 \ee
for which   the following monotonicity formula can be proved:

\begin{lemma}\label{tout}
Let $\widetilde \e(s,y)$ be a solution of   \eqref{eq:toy} satisfying the orthogonality conditions   \eqref{ortho}. Then,
\begin{enumerate}
  \item Energy conservation at $\widetilde \e$ level:
  \begin{equation}\label{ener}\forall s,\quad (L\widetilde \e(s),\widetilde \e(s))=(L\widetilde \e(0),\widetilde \e(0)).  \end{equation}
   \item Virial estimate 
 \be\label{vir1} \frac d{ds} \int y \widetilde \e^2 = - H(\widetilde \e,\widetilde \e),  \ee
  where
  \begin{equation}\label{vir}
  H(\widetilde \e,\widetilde \e)= \int \left( 3 \widetilde \e_y^2 + \widetilde \e^2 - 5 Q^4 \widetilde \e^2 + 20 y Q' Q^3 \widetilde \e^2\right)\geq \mu_0 \|\widetilde \e(s)\|_{H^1}^2.
    \end{equation}
  \item Mixed Virial and energy-monotonicity estimate: for $B\gg 1$,   $\mu_1>0$,
  \begin{equation}\label{mixed} \frac d{ds} \left[\int    \widetilde \e_y^2 \psi\left(\frac yB\right)+  \widetilde \e^2 \varphi\left(\frac yB\right)- 5Q^4 \widetilde \e^2 \psi\left(\frac yB\right) \right] +\mu_1   \int   \left(  \widetilde \e_y^2  +  \widetilde \e^2\right) \varphi'\left(\frac yB\right)  \leq 0,  \end{equation}
  \begin{equation}\label{coer}
   \int   \widetilde \e_y^2 \psi\left(\frac yB\right)+  \widetilde \e^2 \varphi\left(\frac yB\right)- 5Q^4 \widetilde \e^2 \psi\left(\frac yB\right)   \geq \mu_1  \int    \widetilde \e_y^2 \psi\left(\frac yB\right)+  \widetilde \e^2 \varphi\left(\frac yB\right)   ,
    \end{equation}
  where the smooth functions $\varphi$ and $\psi$ satisfy
  \begin{align}\label{varphi}
 & \varphi(y) = \begin{cases} e^{y} \text{ for } y<-1,\\ 1+y \text{ for } -\frac 12<y<\frac 12 ,\quad
  \varphi'\geq 0 \text{ on } \RR,\\
  y \text{ for } y >1,   \end{cases}
    \\
  \label{psi}
 & \psi(y) = \begin{cases} e^{2y} \text{ for } y<-1,\\ 1 \text{ for } y>-\frac 12,\quad
  \psi'\geq 0 \text{ on } \RR.  \end{cases}
    \end{align}
   \end{enumerate}
  \end{lemma}
\begin{proof}
Identities \eqref{ener} and \eqref{vir1} are obtained by direct computations from the equation of $\e$ and classical properties of $L$ (see \cite{W1983}).
The coercivity of $H(\e,\e)$ is proved in \cite{MMjmpa}.
Estimate \eqref{mixed} follows  from direct computations and estimates, and   \eqref{vir}
applied to a suitable localization of $\e$. It  combines in a sharp way
 monotonicity arguments from \cite{MMjmpa} (reminiscent of the Kato smoothing effect \cite{Kato}) and localized   Virial estimates.
The coercivity property \eqref{coer} is a consequence of well-known properties of the operator $L$:  for $\mu>0$,
$$
(\e,Q)=(\e,y\Lambda Q)=(\e,\Lambda Q)=0 \quad \Rightarrow \quad (L\e,\e) \geq \mu \|\e\|_{H^1}^2,
$$
(see \cite{W1983}),
and localization arguments. 
\end{proof}


\subsection{The energy/virial Lyapunov functional}

  
  The mixed energy-Morawetz estimate \fref{mixed} provides both a pointwise control of the boundary term and a space time control of some local norm. At the nonlinear level, this kind of tools may be   delicate to handle due to   localization in space, the main difficulty being to control the nonlinear terms using the sole weighted norms present in \fref{mixed}. This is a   well-known   problem, and the choice of the weights $(\varphi,\psi)$ is here essential. The presence of the drift operator $\lsl\Lambda \e$ in the right-hand side of \fref{eqeps} is an additional difficulty for space localization, which requires the assumption of space decay on the right, i.e. $u_0\in \mathcal A$. For blow up problems, such a strategy based on   mixed energy-virial estimates to control the residual term $\e$ was already used  in \cite{RR2009}, \cite{RS2010} but in a setting where localization in space is   simpler to handle. For (gKdV), the   strategy   developed in \cite{MMR1}  goes as follows.\\
  
  For $B\gg 1$ fixed, define the weighted norms
$$
\mathcal{N}(s)=\int      \e_y^2 \psi\left(\frac yB\right)+    \e^2 \varphi\left(\frac yB\right) , 
$$
and the   nonlinear functional:
$$
\mathcal{F}_i(s) = \int      \e_y^2 \psi\left(\frac yB\right)+    \e^2 (1+\mathcal{J}_i)  \varphi\left(\frac yB\right)- \frac 13 \left( (\e +Q_b)^6 - Q_b^6 - 6 \e Q_b^5\right) \psi\left(\frac yB\right),
$$
where we introduced the following nonlinear lower order corrections\footnote{which should be ignored at first hand.}
$$
\mathcal{J}_i = ( 1-J_1)^{-4i} -1,\quad J_1=(\e,\rho_1),\quad 
\rho_1(y) = \frac 4{\left( \int Q\right)^2}  \int_{-\infty}^y \Lambda Q.
$$
Then, we obtain the following Lyapunov monotonicity which takes the form of a bootstrap bound\footnote{in order to control the nonlinear term.}:

\begin{proposition}[\cite{MMR1}]\label{propasymp}
Assume   on some interval $[0,s_0]$,
\\
{\em (H1)   smallness:}
\be
\label{boundnwe}
\|\e(s)\|_{L^2}+|b(s)|+\mathcal N(s)\leq \kappa^*;
\ee
{\em (H2) comparison between $b$ and $\l$:}
\be
\label{bootassumption}
 \frac{|b(s)|+\mathcal N(s)}{\lambda^2(s)}\leq \kappa^*;
\ee
{\em (H3) $L^2$ weighted bound on the right:}
\be
\label{uniformcontrol}
\int_{y>0}y^{10}\e^2(s,x)dx\leq 10\left(1+\frac{1}{\l^{10}(s)}\right).
\ee
 Then   the following bounds hold on $[0,s_0]$: for $B\gg 1$, $\mu>0$,\\
 {\em (i) Scaling  invariant Lyapunov control:} 
\be
\label{lyapounovconrol}
\frac{d}{ds}  {\mathcal F}_1 +    \mu  \int  \left(\e_y^2 + \e^2\right) \varphi'\left(\frac yB\right) \lesssim  {|b|^{4}} .
\ee
{\em (ii) Scaling weighted $H^1$ Lyapunov control:} 
\be
\label{lyapounovconrolbis}
\frac{d}{ds}\left\{\frac{{\mathcal F}_2}{\l^2}\right\}+  \frac \mu {\l^2}   \int  \left(\e_y^2 + \e^2\right) \varphi'\left(\frac yB\right) \lesssim \frac{|b|^{4}}{\l^2}.
\ee
{\em (iii) Pointwise bounds:}  
\be
\label{controlj}
|J_1|+|J_2|\lesssim \mathcal N^{\frac{1}{2}} ,
\ee
$$\mathcal N  \lesssim \mathcal F_{j}\lesssim \mathcal N  , \ \ j=1, 2.
$$
\end{proposition}


\subsection{Rigidity and selection of the dynamics}


Note that estimate \fref{lyapounovconrolbis} controls the radiation $\e$ independently of   the dynamics   and thus is valid in all possible regimes like (Blow up), (Soliton) or (Exit). In particular,   estimate \fref{lyapounovconrolbis} completely reduces the control of $\e$ to the sole control of the parameter $b$. Combining  this estimate with the finite dimensional evolution equation \fref{contorlloca}   leads to the following rigidity formula for $b$:

\begin{lemma}[Control of the dynamics for $b$, \cite{MMR1}]\label{le:2.7}
Under   assumptions {\rm (H1)-(H2)-(H3)} of Proposition~\ref{propasymp}, 
for all $  0\leq s_1\leq s_2< s_0$,
\be
\label{conrolbintegre}
\left|\frac{b(s_2)}{\lambda^2(s_2)}-\frac{b(s_1)}{\lambda^2(s_1)}\right|\leq
\frac {C^*}{10}\left[\frac{b^2(s_1)}{\lambda^2(s_1)}+\frac{b^2(s_2)}{\l^2(s_2)}+ \frac{\matchal N(s_1)}{\lambda^2(s_1)}\right]
\ee
for some universal constant $C^*>0$.
\end{lemma}

The three scenarios of Theorem \ref{th:3} are   a direct consequence of \fref{conrolbintegre}. Indeed, there are only two possibilities. \\

- Either there exists a time $s_1$ such that at $s_1$, $|b(s_1)|$ dominates $\e$ in the sense that 
\be
\label{event}
\frac{|b(s_1)|}{\l^2(s_1)}\gg \frac{\mathcal N(s_1)}{\l^2(s_1)}.
\ee 
Then from \fref{conrolbintegre} and the a priori bound \fref{boundnwe}-\fref{bootassumption},  this property is propagated at later times and:$$\forall s_2\ge s_1, \ \ \frac{b(s_2)}{\lambda^2(s_2)}\sim\frac{b(s_1)}{\lambda^2(s_1)}=c_0\neq 0.$$ Thus, from \fref{contorlloca}: $$\frac{b}{\l^2}\sim \frac{-\lambda_s}{\l^3}=-\l_t\sim c_0.$$ If $c_0>0$, $\l$ vanishes  in finite time $T$, $\l(t)\sim c_0 (T-t)$ and  this corresponds to (Blow up). If $c_0<0$, $\l$ is growing and so is $b$ so that $u(t)$ is moving away in $L^2$ from the solitary wave, which  is the (Exit) case. Observe that \fref{event} is an open condition on the data, and hence both these regimes are stable. \\

- Or this time $s_1$ does not exist, which means that $$\forall s_1, \ \ \frac{|b(s_1)|}{\l^2(s_1)}\ll\frac{\matchal N(s_1)}{\l^2(s_1)}\ \ \mbox{i.e.}\ \ |b(s_1)|\lesssim \matchal N(s_1).$$ Then the space time bounds on $\e$ eventually lead to the following estimate $$\int_0^{+\infty}\left|\lsl\right|\lesssim\int_0^{+\infty}|b(s)|ds\lesssim \int_0^{+\infty}\matchal N(s)ds<+\infty$$ and thus $$\l(s)\to \l_{\infty}>0\ \ \mbox{as}\ \ s\to +\infty.$$ This is the (Soliton) dynamics which is a threshold regime.


\subsection{Construction of the minimal mass blow up solution}


The construction of the minimal mass solution and the determination of the universal behavior of solutions in the (Exit) regime   follow the same compactness strategy. The minimal blow up solution in Theorem~\ref{th:2} is obtained as the limit of sequences of {\it defocusing} solutions. Indeed, we pick a sequence of {\it well-prepared} initial data $$u_n(0)=Q_{b_n(0)}, \ \ b_n(0)=-\frac{1}{n}$$ which by construction have subcritical mass $$\|u_n(0)\|_{L^2}-\|Q\|_{L^2}\sim \frac{c}{n}.$$ Such solutions are necessarily in the (Exit) regime of Theorem \ref{th:3} and we denote by $t_n^*$ the corresponding exit time. Moreover, we have from \cite{MMR1} (see also the formal discussion of Section~3) a precise description of the flow in the time interval $ [0,t^*_n]$; in particular, we know that  the solution admits a decomposition 
\be
\label{cneoneoncoe}
u_n(t,x)=\frac{1}{\l_n^{\frac 12}(t)}(Q_{b_n(t)}+\e_n)\left(t,\frac{x-x_n(t)}{\l_n(t)}\right)
\ee where, at the leading order, $(b_n,\l_n)$ behaves as follows $$\frac{b_n(t)}{\l^2_n(t)}\sim b_n(0)=-\frac{1}{n}, \ \ (\lambda_n)_t\sim- b_n(0),$$ 
\be
\label{cnekneneo}
\lambda_n(t)\sim 1-b_n(0)t, \ \ b_n(t)\sim b_n(0)\l_n^2(t).
\ee
The (Exit) time $t_n^*$ is the one for which the solution moves strictly away from the solitary wave which in our setting is equivalent to $$b_n(t_n^*)\sim-\alpha^*,$$ with $\alpha^*$ independent of $n$. This  allows us to compute $t_n^*$ and show using \fref{cnekneneo} that the solution defocuses: $$\l^2_n(t_n^*)\sim \frac{b_n(t_n^*)}{b_n(0)}\sim n\alpha^* \ \  \mbox{as}\ \ n\to +\infty.$$ 

Next, we renormalize the flow at $t_n^*$, considering the solution of (gKdV) defined by   $$v_n(\tau,x)=\l_n^{\frac12}(t_n^*)u_n(t_\tau,\l_n(t_n^*)x+x_n(t_n^*)),\ \ t_{\tau}=t_n^*+\tau \l_n^3(t_n^*).$$ From direct computations, $v_n$ admits a decomposition $$v_n(\tau,x)=\frac{1}{\l_{v_n}^{\frac 12}(\tau)}(Q_{b_{v_n}}+\e_{v_n})\left(\tau,\frac{x-x_{v_n(\tau)}}{\lambda_{v_n(\tau)}}\right)$$ with from the symmetries of the flow 
$$\l_{v_n}(\tau) = \frac {\l_n(t_\tau)}{\l_n(t_n^*)},  \ \ 
x_{v_n}(\tau)= \frac  {x_n(t_\tau)-x_n(t_n)}{\l_n(t_n^*)} ,\ \ 
b_{v_n}(\tau) = b_n(t_\tau),\ \ 
\e_{v_n}(\tau) = \e_n(t_\tau).
$$
The renormalized parameters can be computed at the main order using \fref{cnekneneo}:
\bee
\l_{v_n}(\tau) & \sim &   \frac{1}{\l_n(t_n^*)}\left[1-b_n(0)(t_n^*+\tau\l_n^3(t_n^*))\right]\\
& \sim & \frac{1}{\l_n(t_n^*)}\left[\l_n(t_n^*)-\tau b_n(0)\l_n^3(t_n^*)\right]  
 \sim   1-\tau b_n(t_n^*)\sim1+\tau\alpha^*.
\eee
Observe that the law of $\l_{v_n}(\tau)$ at this order does not depend on $n$, which is a  remarkable and  decisive property in this approach.
Letting $n\to +\infty$, we   extract a weak limit in $H^1$ $v_n(0)\rightharpoonup v(0)$ such that the corresponding solution $v(\tau)$ to (gKdV)   blows up backwards at some finite time $\tau^*\sim -\frac{1}{\alpha^*}$ with the blow up speed $\lambda_v(\tau)\sim \tau-\tau^*$ as expected. 
Note that the extraction of the weak limit   requires uniform estimates on the residual radiation $\e_{v_n}$. Here it is essential that the set of data $u_n(0)$ is {\it well-prepared}, as this implies uniform bounds for $\e_{v_n}(0)=\e_{u_n}(t_n^*)$ in $H^1$ and allows us to use the $H^1$ weak continuity of the flow in the limiting process.
Finally,  by the weak convergence, one obtains $\|v\|_{L^2}\leq \|Q\|_{L^2}$, but since the solution $v(\tau)$ blows up in finite time, $\|v\|_{L^2}=\|Q\|_{L^2}$.\\

For uniqueness, we refer the reader to  \cite{MMR2}.

\subsection{Solutions in the (Exit) regime.} Now, we prove the universality of $S_{\rm KdV}$ as an attractor in the (Exit) case. For this, we consider a sequence  of data $(u_0)_n$ with $\|(u_0)_n\|_{L^2}\to \|Q\|_{L^2}$ as $n\to +\infty$ such that the corresponding solution to (gKdV) is in the (Exit) regime. We write   the solution at the (Exit) time in the form \fref{cneoneoncoe}, renormalize the flow and   aim at extracting a weak limit as $n\to +\infty$ as before. The strategy  of the proof is similar to the construction of the minimal mass solution, except that since the data is not well-prepared, no uniform $H^1$ bound on $v_n(0)$ can be obtained. To get around, we use two additional ingredients:
\begin{enumerate}
\item a concentration compactness argument on sequences of solutions in the critical $L^2$ space in the spirit of \cite{KM} using the tools developed in \cite{Shao} and \cite{KKSV} for the Airy group. This allows us to extract a non trivial weak limit with suitable dynamical controls; 

\item refined {\it local} $H^1$ bounds on $v_n(\tau)$ in order to ensure that the $L^2$ limit of this sequence actually belongs to $H^1$. 
\end{enumerate}
Hence the weak limit is a minimal mass $H^1$ blow up element, and by the uniqueness statement of Theorem \ref{th:2}, the limit is $S_{\rm KdV}$ up to the symmetries of the equation, which provides the final conclusion of Theorem \ref{th:3}.

\end{document}